\documentclass{amsart}
\usepackage{amsmath}
\usepackage{amsthm,amsfonts,amssymb,mathrsfs}
\usepackage{epic,eepic}
\usepackage{yfonts}
\usepackage{nicefrac}
\usepackage{paralist,enumerate}
\usepackage[all]{xy}

\usepackage{xcolor,graphicx}
\usepackage[mathscr]{euscript}

\theoremstyle{definition}
\newtheorem{mtheorem}{Theorem}
\newtheorem{theorem}{Theorem}[section]
\newtheorem{definition}[theorem]{Definition}
\newtheorem{lemma}[theorem]{Lemma}
\newtheorem{proposition}[theorem]{Proposition}
\newtheorem{conj}[theorem]{"Conjecture"}
\newtheorem{corollary}[theorem]{Corollary}
\newtheorem*{theorem*}{Theorem}

\theoremstyle{remark}
\newtheorem{remark}[theorem]{Remark}

\usepackage{comment}

\def\Log{{\rm Log\,}} 

\usepackage{bbm}

\newcommand{\RR}{\mathbb{R}}
\newcommand{\ZZ}{\mathbb{Z}}
\newcommand{\CC}{\mathbb{C}}

\newcommand{\TS}{\mathscr{T}}
\newcommand{\Ccal}{\mathcal{C}}
\newcommand{\FC}{\mathcal{F}}

\begin{document}

\title{Convexity of complements of tropical varieties, and approximations of currents}
\date{}
\author{Karim Adiprasito and Farhad Babaee}
\thanks{Karim Adiprasito is supported by ERC StG 716424 - CASe and ISF Grant 1050/16. Farhad Babaee is supported by the SNSF:PP00P2~150552$\slash$1 Grant.}

\begin{abstract}
The goal of this note is to affirm a local version of conjecture of Nisse--Sottile \cite{Nisse-Sottile} on higher convexity of complements of tropical varieties, while providing a family of counter-examples for the global Nisse--Sottle conjecture in any codimension and dimension higher than 1. Moreover, it is shown that, surprisingly, this family also provides a family of counter-examples for the generalized Hodge conjecture for positive currents in these dimensions, and gives rise to further approximability obstruction. 

\end{abstract}
\maketitle

\section{Introduction}

Following a classical theorem of Aumann \cite{Aumann}, compact convex sets can be characterized using a homological condition on section by linear subspaces of a fixed dimension. 

Gromov \cite{Gromov} intuited from this and subsequent results that such homological conditions for non-compact sets should be thought of as lying on the same spectrum as the Lefschetz section property for smooth compact varieties. Hence, a natural notion of \emph{global $k$-convexity}, for short $k$-convexity, is that any nontrivial homology  $(k-1)$-dimensional cycle in the section of the given set with an affine $k$-plane extends to a nontrivial global homology cycle. He related this to curvature properties of submanifolds, in particular providing a necessary and sufficient condition for a smooth submanifold of Euclidean $n$-dimensional space of dimension $k$ to be $(n-k)$-convex.

The advent of the study of algebraic aspects of tropical and log-geometry brought with it a desire to sink also Lefschetz' topological harpoon into tropical varieties and amoebas and their complements. {Recall that an Archimedean amoeba (or simply an amoeba) of an algebraic variety of $(\CC^*)^n$ is its image under the coordinatewise valuation with $\log|.|$ whereas its non-Archimedean amoeba is its image under a non-Archimedean coordinatewise valuation. Non-Archimedean amoebas can be, in fact, seen as Hausdorff limit of Archimedean amoebas, and  have the structure of a tropical variety, and are also known as realizable tropical varieties. Henriques thereby conjectured that complements of dimension $k$ amoebas in $\RR^n$ are $(n-k)$-convex  \cite{Henriques}}, and was followed by Nisse--Sottile \cite{Nisse-Sottile} who conjectured the $(n-k)$-convexity for complements of all $k$-dimensional tropical varieties of $\RR^n$. Henriques showed a weak form of his conjecture, while Nisse--Sottile showed that the convexity results for Archimedean amoebas pass over the limit, and extend to non-Archimedean amoebas. In consequence, the convexity results of amoebas in the complete intersection case proved by Bushueva-Tsikh in \cite{Bushueva-Tsikh}, as well as the \emph{positive convexity} (See Definition~\ref{pos-convex}) of amoebas proved by \cite{Henriques} also hold for non-Archimedean amoebas. Adiprasito--Bj\"orner in \cite{AB} showed that the smooth $k$-dimensional tropical varieties, introduced by Mikhalkin, and their complements, indeed satisfy the Lefschetz theorem, and are in particular $(n-k)$-convex.

The first main theorem of this paper establishes a local {and a weaker} version (Definition~\ref{def-local}) of the Nisse-Sottile conjecture (See Theorem~\ref{no-caps-cycles} below). {The local convexity of complements of amoebas was previousely shown by Mikhalkin in \cite{Mikh-amoeba}, and in a more general situation by Rashkovskii in \cite{Rash-amoeba} using a powerful result of Fornaess--Sibony in \cite{Fornaess-Sibony}. Our approach to prove Theorem~\ref{intro-local-cov} is similar to that of Rashkovskii.}
\begin{mtheorem}\label{intro-local-cov}
The complement of a tropical variety of dimension $k$ in $\RR^n$ is locally $(n-k)$-convex.
\end{mtheorem}

We then focus entirely on tropical varieties, establishing the following negative result (Theorem~\ref{thm-examples}):
\begin{mtheorem}\label{intro-B}
For any pair of integers $n\geq 4,$ and $2\leq k \leq n-2,$ there is a tropical variety of dimension~$k$ in $\RR^n$ whose complement is not globally $(n-k)$-convex.
\end{mtheorem}
This contrasts the earlier result of Nisse-Sottile in codimension one. In light of the results in by the first author and Huh-Katz in \cite{AHK} we will see that the above examples do not satisfy the Hodge-Riemann relations, and combined with the results of the second author with Huh in \cite{BH} they give rise to a family of counter-examples for the generalized Hodge conjecture for positive currents (See Section 6 for details). Namely:  

\begin{mtheorem}
For any pair of integers $n\geq 4,$ and $2\leq k \leq n-2,$ 
there is an $n$-dimensional smooth projective toric variety $X$ and a $(k,k)$-dimensional positive closed current $\mathscr{T}$ on $X$ with the following properties:
\begin{enumerate}[(1)]
\item The cohomology class of $\mathscr{T}$ satisfies
\[
\{\mathscr{T}\} \in H^{2(n-k)}(X,\mathbb{Z})\slash\textrm{tors}~ \cap H^{n-k,n-k}(X).
\]
\item The current $\mathscr{T}$ is not a weak limit of currents of the form
\[
\lim_{i \to \infty} \mathscr{T}_i, \quad \mathscr{T}_i=\sum_j \lambda_{ij} [Z_{ij}],
\]
where $\lambda_{ij}$ are nonnegative real numbers and $Z_{ij}$ are $k$-dimensional subvarieties  in $X$.
\end{enumerate}
\end{mtheorem}

As the higher convexity properties of open sets passes over the Hausdorff limit\footnote{Throughout this article, by Hausdorff limit, we mean Hausdorff limit with respect to compact subsets of $\RR^n$.} and the global and local notions of convexity are equivalent for domains with smooth boundary, non-$(n-k)$-convex examples in Theorem B gives rise to obstruction for exhaustion of \emph{$k$-pseudoconvex} domain (Definition~\ref{def-pseudocon}) with $k$-pseudoconvex domains which have smooth boundaries (Compare to \cite{Diederich-Fornaess}). This observation has the following consequence (Theorem~\ref{thm-non-moli}):

\begin{mtheorem}
For any pair of integers $n\geq 4,$ and $2\leq k \leq n-2,$ there is a closed positive current $\mathscr{T}\in \mathscr{D}'_{k,k}((\CC^*)^n),$ for which there is no family of mollifiers $\mu_t$ with 
\begin{itemize}
\item[(i)] $\text{Supp} (\mu_t \star \mathscr{T})$ converges to $\text{Supp} (\mathscr{T})$ in Hausdorff sense as $t\to 0$, and
\item [(ii)] $\text{Supp} (\mu_t\star \mathscr{T})$ have smooth boundaries. 
\end{itemize}
\end{mtheorem}

\subsection*{Acknowledgments}
{The authors thank the referees for the fruitful comments.} We thank the CRM Ennio de Giorgi for its hospitality that enabled much of this work. Additionally, we are grateful to Emanuele Delucchi, Omid Amini, Romain Dujardin, Charles Favre, and Pierre Schapira for the encouraging discussions, and their support, and we are thankful to Nessim Sibony for the communications and references. 
\section{Notions of convexity in $\RR^n$}

There are several notions of convexity in $\RR^n.$ A global notion of convexity was considered in \cite{Gromov}, and \cite{Henriques}: 

\begin{definition}\label{def-global}
A subset $X\subseteq\RR^n$ is called \emph{globally $k$-convex} if the homomorphisms of reduced homology groups
$$
i_L:\widetilde{H}_{k-1}(L \cap X)\rightarrow \widetilde{H}_{k-1}(X), 
$$
is injective, for all affine $k$-planes $L\subseteq \RR^n.$ 
\end{definition}

By Morse theory one can see that if $X\subseteq \RR^n$ is an open domain with smooth boundary $\partial X,$ then the $k$-convexity of $X,$ is equivalent to $\partial X$ having at most $(k-1)$-negative principal curvatures with respect to the inward normal which in turn implies non-existence of singularity of Morse index $k$ (See \cite[Section \nicefrac{1}{2}]{Gromov}).  Existence of $k$-negative principal directions, however, simply implies existence of a \textit{supporting $k$-cap} for $A := \RR^n \setminus X$, which can be defined for non-smooth subsets as well.  

\begin{definition}[\cite{Mikh-amoeba}]\label{def-caps}
An open round disk $B\subseteq L$ of radius $\delta>0$, in an affine $k$-plane $L\subseteq \RR^n,$ is called a supporting $k$-cap ($k$-cap for short) for $A\subseteq \RR^n$, if 
\begin{itemize}
\item $B\cap A$ is nonempty and compact;
\item There exists a vector $v \in \RR^n$ such that the translation of $B$, by $\epsilon v$ is disjoint from $A$ for sufficiently small $\epsilon >0.$
\end{itemize}
\end{definition}

\begin{definition}\label{def-local}
A subset $X\subseteq\RR^n$ is called \emph{locally $k$-convex} if the complement $A=\RR^n \setminus X,$ has no supporting $k$-caps. 

\end{definition}

Unlike the smooth case, we will see in Section~\ref{sec-counterexample} that these two notions are not equivalent in general. An intermediate notion of higher convexity was considered by Henriques:
\begin{definition}[\cite{Henriques}]\label{pos-convex}
A subset $X\subseteq\RR^n$ is called positively $k$-convex, if for any $k$-plane $L,$ the if the kernel of the homomorphisms of reduced homology groups
$$
i_L:\widetilde{H}_{k-1}(L \cap X)\rightarrow \widetilde{H}_{k-1}(X), 
$$
does not contain any \textit{positive} $k$-chain. Here, a positive chain is a positive sum of $k$-cycles which have the positive orientation induced by a fixed orientation from $L$.  
\end{definition}

The following proposition partially compares the {aforementioned} notions of convexity and hints that the supporting caps fails to be a good generalization of notion of Morse index for stratified spaces. 
 
\begin{proposition}\label{local-global}
Let  $X \subseteq \RR^n,$ 
\begin{enumerate}
\item  $X$ is $k$-convex  $\implies$ $X$ is positively $k$-convex $\implies$ $X$ is locally $k$-convex. 
\item  If $X$ has a smooth boundary then, for $X$ the three notions are equivalent. 
\item In general, locally $k$-convexity does not imply global the $k$-convexity.
\end{enumerate} 
\end{proposition}
\begin{proof}
The first implication follows easily from the definition. The second is an application of Morse theory which we discussed before Definition~\ref{def-caps} (See also \cite[Section 1/2]{Gromov}). For the last statement we
will prove in Theorem~\ref{no-caps-cycles} that every $k$-dimensional tropical cycle in $\RR^n$ lacks $(n-k)$-caps, but in Section~\ref{sec-counterexample} for any pair of integers $n\geq 4,$ and $2\leq k \leq n-2,$ 
we will construct a $k$-dimensional tropical variety $\widehat{F}_{k,n}\subseteq \RR^n$ whose complement is not globally $(n-k)$-convex.
\end{proof}

\section{Pseudoconvexity in $\CC^n$}
We recall Rothstein's notion of pseudoconvexity in $\CC^n$, which is invariant under biholomorphic mappings. The notion is defined using the \emph{Hartogs figures}: Given $0<q<n$ and $\alpha,\beta \in (0,1),$ the set
$$
H= \{(z,w)\in \CC^{n-q} \times \CC^q: ||z||_{\infty}<1, ||w||_{\infty}<\alpha \text{ or }  \beta<||z||_{\infty}<1, ||w||_{\infty}<1\}
$$
is called an $(n-q,q)$-Hartogs figure, where $||z||_{\infty}=\max_j|z_j|.$ Note that the convex hull $\widehat{H}$ is simply a polydisc. 

\begin{definition}[Rothstein's $q$-pseudoconvexity]\label{def-pseudocon}
An open subset $\Omega$ of a complex $n$-dimensional manifold $M$ is said to be $q$-pseudoconvex in $M$ if for any $(n-q,q)$-Hartogs figure $H$, and a biholomorphic map $\Phi:H\rightarrow M,$ the condition $\Phi(H)\subseteq \Omega$ implies $\Phi(\widehat{H})\subseteq \Omega.$  In this case, $M\setminus \Omega$ is called $q$-pseudoconcave in $M.$
\end{definition}

Pseudoconvexity in $\CC^n$ is related by the following lemma to the local convexity in $\RR^n.$

\begin{lemma}[{\cite[Proposition 2.2]{Rash-amoeba}}]\label{Rash-caps}
Let $\Gamma$ be a closed subset of a convex open set $D\subseteq \RR^n.$ If the tube set $\RR^n+ i\Gamma$ is $k$-pseudoconcave in the tube domain $\RR^n+iD,$ then $\Gamma$ has no supporting $(n-k)$-caps. 
\end{lemma}
\begin{proof}
Since by definition supporting caps are open, a $(n-k)$-cap $B$ for $\Gamma$ gives rise to a $(k,n-k)$ Hartogs figure $H \subseteq\RR^n +i( D
\setminus \Gamma),$ such that $\widehat{H}\cap \Gamma \not=\emptyset,$ hence $\RR^n +i( D\setminus \Gamma)$ cannot be $k$-pseudoconvex.
\end{proof}

\begin{corollary}\label{log-pseudo}
Let  $\Gamma \subseteq D$ be as in the preceding proposition. If the set $\Log^{-1}(\Gamma)$ is $k$-pseudoconcave in the domain $\Log^{-1}(D),$ then $\Gamma$ has no supporting $(n-k)$-caps.  
\end{corollary}
\begin{proof}
The map $(-i\log, \dots, -i\log): \CC^n \rightarrow \CC^n,$ is a biholomorphism in the unit complex disk, and by invariance of the notion of pseudoconvexity under bihomolomorphisms, we see that $\Log^{-1}(\Gamma)$ is $k$-pseudoconcave in $\Log^{-1}(D)$ if and only if $\RR^n + i \Gamma$ is $k$-pseudoconcave in $\RR^n+ iD.$ The result therefore follows from Lemma~\ref{Rash-caps}.
\end{proof}

\section{Pseudoconcavity of support of currents}

Let $X$ be a complex manifold of dimension $n$. If $k$ is a nonnegative integer, we denote by $\mathscr{D}^{k}(X)$ the space of smooth complex differential forms of degree $k$ with compact support, endowed with the inductive limit topology. The space of currents of dimension $k$ is the topological dual space $\mathscr{D}'_{k}(X)$, that is, the space of all continuous linear functionals on $\mathscr{D}^{k}(X)$:
\[
\mathscr{D}'_{k}(X):=\mathscr{D}^{k}(X)'.
\]
The pairing between a current $\mathscr{T}$ and a differential form $\varphi$ will be denoted $\langle \mathscr{T},\varphi \rangle$. A $k$-dimensional current $\mathscr{T}$ is a \emph{weak limit} of a sequence of $k$-dimensional currents $\mathscr{T}_i$  if 
\[
\lim_{i \to \infty} \langle \mathscr{T}_i, \varphi \rangle = \langle \mathscr{T}, \varphi \rangle \ \ \text{for all $\varphi \in \mathscr{D}^{k}(X)$}. 
\]
 
The exterior derivative of a $k$-dimensional current $\mathscr{T}$ is the $(k-1)$-dimensional current  $d\mathscr{T}$ defined by
\[
\langle d\mathscr{T},\varphi \rangle = (-1)^{k+1} \langle \mathscr{T}, d\varphi \rangle,  \quad \varphi \in \mathscr{D}^{k-1}(X).
\]

The current $\mathscr{T}$ is \emph{closed} if its exterior derivative vanishes. By duality  one simply has the following decompositions for the bidegree and bidimension
\[
\mathscr{D}^k(X)=\bigoplus_{p+q=k} \mathscr{D}^{p,q}(X), \quad \mathscr{D}'_k(X)=\bigoplus_{p+q=k} \mathscr{D}'_{p,q}(X).
\]

The space of smooth differential forms of bidegree $(p,p)$ contains the cone of positive differential forms. By definition, a smooth differential $(p,p)$-form $\varphi$ is \emph{positive} if
\[
\text{$\varphi(x)|_S$ is a nonnegative volume form for all $p$-planes $S \subseteq T_x X$ and $x \in X$}.
\]
Dually, a current $\mathscr{T}$ of bidimension $(p,p)$ is \emph{positive}  if 
\[
\langle \mathscr{T}, \varphi \rangle \ge 0 \ \ \text{for every positive differential $(p,p)$-form $\varphi$ on $X$}.
\]
Integrating along complex analytic subsets of $X$ provides an important class of positive currents on $X$.
If $Z$ is a $p$-dimensional complex analytic subset of $X$, then the \emph{integration current} $[Z]$ is the $(p,p)$-dimensional current defined by integrating over the smooth locus
\[
\big\langle [Z], \varphi \big\rangle = \int_{Z_{\text{reg}}} \varphi, \quad \varphi \in \mathscr{D}^{p,p}(X).
\]

One has the following important theorem of Fornaess--Sibony.

\begin{theorem}[{\cite[Corollary 2.6]{Fornaess-Sibony}}]\label{FS-concave}
The support of a positive closed current of bidimension $(k,k)$ on a complex manifold $M$ is $k$-pseudoconcave in $M.$
\end{theorem}

In the next subsection we explain the construction of \emph{tropical currents} from \cite{Babaee}. These currents are closed and positive and have support of the form $\Log^{-1}(\{\text{tropical variety}\}).$ Our exposition follows \cite{BH}.

\subsection{A current associated to a tropical cycle}
We first recall the definition of a tropical variety. 

For a given polyhedron $\sigma$ let $\text{Aff}(\sigma)$ be the affine span of $\sigma,$ and $H_{\sigma}$ be the translation of $\text{Aff}(\sigma)$ to the origin. Now assume that $\tau$ is a codimension $1$ face of a $p$-dimensional rational polyhedron $\sigma$ and $u_{\sigma / \tau }$ be the unique outward generator of the one dimensional lattice $(H_{\sigma}\cap \ZZ^n) / (H_{\tau}\cap \ZZ^n).$ 

\begin{definition}\label{BalancingCondition}
A $p$-dimensional weighted complex $\mathcal{C}$ satisfies the \emph{balancing condition} at $\tau$ if
\[
\sum_{\sigma\supset \tau} \text{w}_{\mathcal{C}}(\sigma) u_{\sigma/\tau}=0  
\]
in the quotient lattice $\ZZ^n / (H_{\tau}\cap \ZZ^n)$, where the sum is over all $p$-dimensional cells $\sigma$ in $\mathcal{C}$ containing $\tau$ as a face.
A weighted complex is \emph{balanced} if it satisfies the balancing condition at each of its codimension $1$ cells.
\end{definition}

A \emph{tropical variety} is a positive and balanced weighted complex with finitely many cells, and a \emph{tropical current} is the current associated to a tropical variety.

We next define the definition of tropical currents. Let $\mathbb{C}^*$ be the group of nonzero complex numbers. The \emph{logarithm map} is the homomorphism
\[
\text{log}: \mathbb{C}^* \longrightarrow \mathbb{R}, \qquad z \longmapsto \log |z|,
\]
and the \emph{argument map} is the homomorphism
\[
\text{arg}: \mathbb{C}^* \longrightarrow S^1, \qquad z \longmapsto z/|z|.
\]
The argument map splits the exact sequence
\[
\xymatrix{
0\ar[r]&S^1\ar[r]&\mathbb{C}^*\ar[r]^{\text{log}}&\mathbb{R}\ar[r]&0,
}
\]
giving polar coordinates to nonzero complex numbers. Let $N$ be a finitely generated free abelian group. 
There are Lie group homomorphisms
\[
\xymatrix{
&T_N \ar[dr]^{\text{arg}\otimes_\mathbb{Z}1} \ar[dl]_{-\text{log}\otimes_\mathbb{Z}1}&\\
N_\mathbb{R}&&S_N,}
\]
called the \emph{logarithm map} and the \emph{argument map} for $N$ respectively, where 
\begin{eqnarray*}
T_N&:=&\text{the complex algebraic torus $\mathbb{C}^* \otimes_\mathbb{Z} N$,}\\
S_N&:=&\text{the compact real torus $S^1 \otimes_\mathbb{Z} N$,}\\
N_\mathbb{R}&:=&\text{the real vector space $\mathbb{R} \otimes_\mathbb{Z} N$.}
\end{eqnarray*}

For a rational linear subspace $H \subseteq \RR^n$ the following exact sequences:
\[
 \xymatrix{
 0 \ar[r]& {H \cap \mathbb{Z}^n} \ar[r]& \mathbb{Z}^n \ar[r] &{\mathbb{Z}^n/(H \cap \mathbb{Z}^n)} \ar[r] & 0
}
\]

\[
 \xymatrix{
 0 \ar[r]& S_{H \cap \mathbb{Z}^n} \ar[r]& (S^1)^n = S^1\otimes_{\mathbb{Z}}  \mathbb{Z}^n \ar[r] &S_{\mathbb{Z}^n/(H \cap \mathbb{Z}^n)} \ar[r] & 0
}
\]

Define 
\[
\pi_H:\xymatrix{ \text{Log}^{-1}(H) \ar[r]^{\quad\text{Arg}}& (S^1)^n \ar[r]&  S_{\mathbb{Z}^n/(H \cap \mathbb{Z}^n)}}.
\]
\begin{definition}
Let $\mu$ be a complex Borel measure on $S_{\mathbb{Z}^n/(H \cap \mathbb{Z}^n)}$.
We define a $(p,p)$-dimensional closed current $\mathscr{T}_H(\mu)$ on $(\mathbb{C}^*)^n$ by
\[
\mathscr{T}_H:=\int_{x \in S_{\mathbb{Z}^n/(H \cap \mathbb{Z}^n)}} \big[\pi_H^{-1}(x)\big] \ d\mu(x),
\]
where $\mu$ is the Haar measure on $S_{\mathbb{Z}^n/(H \cap \mathbb{Z}^n)}$ 
\[
\int_{x \in S_{\mathbb{Z}^n/(H \cap \mathbb{Z}^n)}} d\mu(x)=1,
\]. 
\end{definition}

Let $A$ be a $p$-dimensional affine subspace of $\mathbb{R}^n$ parallel to the linear subspace $H$.
For $a \in A$,  there is a commutative diagram of corresponding translations
\[
\xymatrix{
\text{Log}^{-1}(A) \ar[d]_{\text{Log}} \ar[r]^{e^{a}}& \text{Log}^{-1}(H) \ar[d]^{\text{Log}}\\
A \ar[r]^{-a}& H.
}
\]
We define a submersion $\pi_{A}$ as the composition
\[
\pi_{A}:\xymatrix{\text{Log}^{-1}(A) \ar[r]^{e^{a} \ } & \text{Log}^{-1}(H) \ar[r]^{\pi_{H} \ \ }&S_{\mathbb{Z}^n/(H \cap \mathbb{Z}^n)}.}
\]

\begin{definition}
Let $\mathcal{C},$ be a weighted rational polyhedral complex of dimension $p$. The tropical current $\mathscr{T}_{\mathcal{C}}$ associated to ${\mathcal{C}}$ is given by 
$$
\mathscr{T}_{\mathcal{C}}= \sum_{\sigma} w_{\sigma} ~ \mathbbm{1}_{\Log^{-1}(\sigma^{\circ})}\mathscr{T}_{\text{aff}(\sigma)} ,
$$

where the sum runs over all $p$ dimensional cells $\sigma$ of $\mathcal{C}.$
\end{definition}

\begin{lemma}[{\cite[Theorem 3.1.8]{Babaee}}]
A weighted complex $\mathcal{C}$ is balanced if and only if $\mathscr{T}_{\mathcal{C}}$ is closed.
\end{lemma}

Now we can simply answer the question about local convexity of complements of tropical cycles. 
\begin{theorem}\label{no-caps-cycles}
The complement of a tropical variety of dimension $k$ in $\RR^n$ is locally $(n-k)$-convex.
\end{theorem}
\begin{proof}
Since $\mathcal{C}$ is a tropical cycle of dimension $k$, the associated tropical current $\mathscr{T}_{\mathcal{C}}$ is a positive current of bidimension $(k,k)$, therefore by Theorem~\ref{FS-concave} its support, $\Log^{-1}(\mathcal{C}),$ is $k$-pseudoconcave in $(\CC^*)^n$, and by Corollary~\ref{log-pseudo}, $\mathcal{C}$ has no supporting $(n-k)$-caps.  
\end{proof}


\section{A recipe for counter-examples}\label{sec-counterexample}
In this section we discuss the global case of the Nisse--Sottile conjecture, for which we provide a family of counter-examples to treat all dimensions where the conjecture is open. The following theorem collects the constructions of this section.

\begin{theorem}\label{thm-examples}
For any pair of integers $n\geq 4,$ and $2\leq k \leq n-2,$ there is a tropical variety of dimension~$k$ in $\RR^n$ (in fact, a fan) whose complement is not $(n-k)$-convex.
\end{theorem}


The rest of this section is devoted to  construction of a $k$-dimensional balanced fan $\widehat{\mathcal{F}}_{k,n}\subseteq \RR^{n+k-2},$ for $k\geq 2$ and $n\geq k+2$, which is not globally $(n-2)$-convex. In order to achieve this we construct a $2$-dimensional balanced fan $\widehat{\mathcal{F}}_{2,n}\subseteq \RR^n ,$ and let 
$$
\widehat{\mathcal{F}}_{k,n}:= \RR^{k-2} \times \widehat{\mathcal{F}}_{2,n} \subseteq \RR^{n+k-2}.
$$

Consider oriented hyperplanes $L,I$. Let $\widetilde{L}^+:= L\cap I^+$ and $\widetilde{\ell}:=\partial \widetilde{L}^+= L\cap I.$ Let $P$ denote an $n$-polytope such that 
\begin{itemize}
\item All faces that do not intersect $\widetilde{\ell}$ are simplicial, and
\item The normal fan $\Sigma$ of $P$, restricted to all remaining faces, is a subdivision of $\widetilde{\ell}$.
\end{itemize}
Such a polytope is easily constructed, for instance as a \textit{Cayley polytope} of $n-1$ generic polygons in $\mathbb{R}^2=\widetilde{\ell}^\bot \subset \mathbb{R}^n$, \textit{i.e.}, consider $e_1, \cdots, e_{n-1}$ a minimal set of vectors in $\widetilde{\ell}$ that sum to $0$, a \textit{circuit}, and let $P_1, \cdots, P_{n-1}$ generic polygons, then 
\[P= \mathrm{conv}\{P_i+e_i\}\]
satisfies the conditions above, see also Figure~\ref{fig:Cayley}. In fact, what we are studying is the 1-skeleton of the Cayley-complex, introduced by the first author in \cite{AS}.

 Now, we wish to construct a nonnegative weight on the $2$-skeleton of the fan $\Sigma$ over $P$. For this, we can triangulate the non-simplicial faces of $\Sigma$, obtaining a simplicial fan $\Sigma'$. Now, we can consider the polytope $P$ as a conewise linear function that induces a nef divisor $p$ in the toric variety $X_{\Sigma'}$. Then $p$ acts on Chow cohomology, and therefore Minkowski weights, as described in \cite{Fulton-Sturmfels}. Consider the weight $p^{n-2} \mu$, where $\mu$ is the unique positive Minkowski $n$-weight on $\Sigma'$. This is the desired nonnegative weight on the $2$-skeleton of $\Sigma$, and we denote its support by~$\mathcal{F}$.

\begin{figure}
\begin{center}
\includegraphics[scale=0.13]{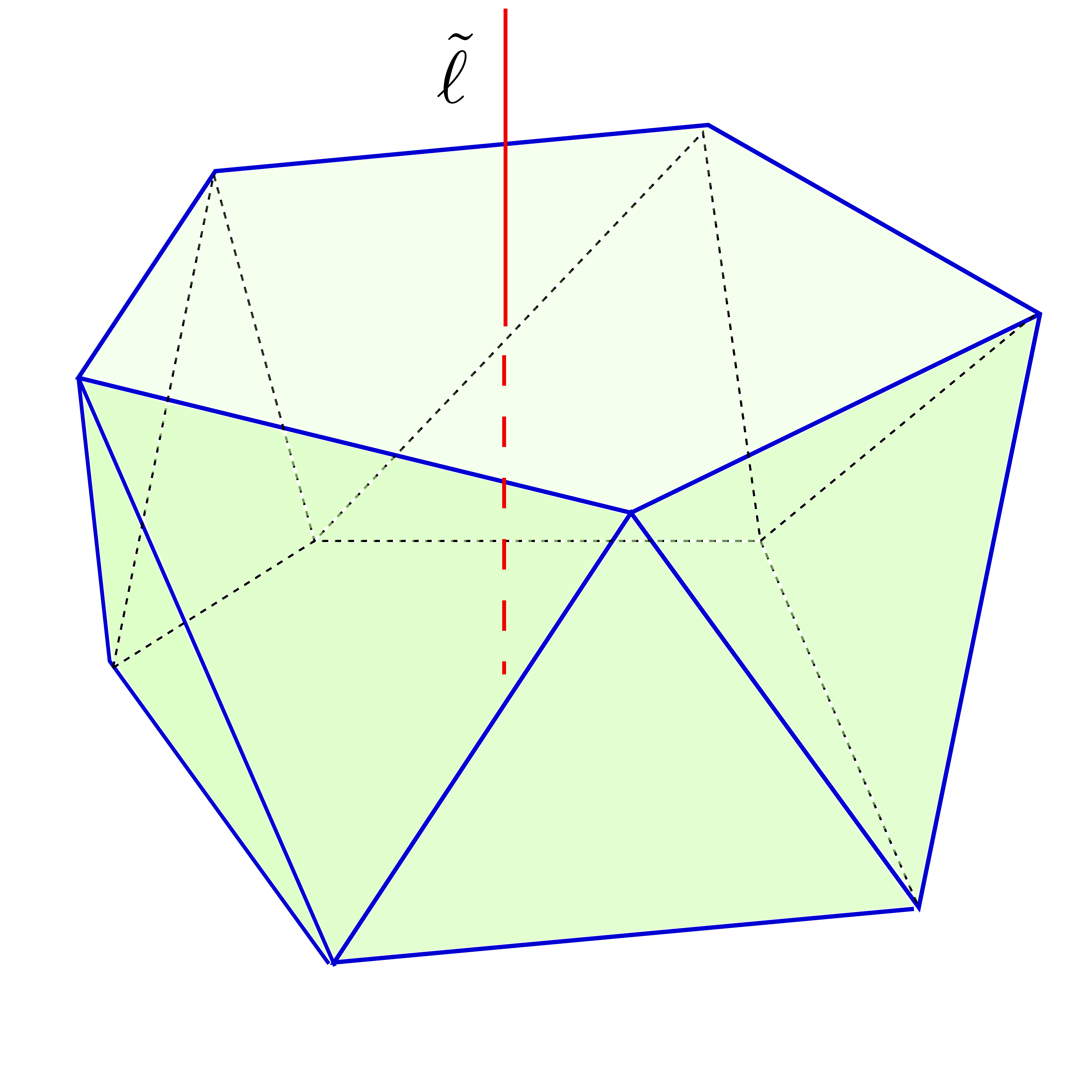}
\caption{\footnotesize Up to now, the construction works even in dimension $3$, giving a positive weight on the $1$-skeleton of a pentagonal antiprism. Only when we perturb the double cover into general position do we have to restrict to dimension $\ge 4$.}
\label{fig:Cayley}
\end{center}
\end{figure}

We may assume, by applying an isometry of $\mathbb{R}^n$, that $L$ and $\mathcal{F}$ intersect transversally (\textit{i.e.}, the linear spans of faces of $\mathcal{F}$ with $L$ are transversal).

Consider $X=\mathcal{F}\cap S^{n-1}$, as a $1$-dimensional ``tropical cycle'' in $S^{n-1}$. Since $X$ does not intersect $\widetilde{\ell}$, the orthogonal projection along $\widetilde{\ell}$ provides a surjective map
\[\pi_1(X)\ \longrightarrow\ \pi_1(\mathbb{R}^2\setminus \{\mathbf{0}\})\ \cong\ \pi_1(S^1)\ \cong\ \mathbb{Z}\]
Consider the Riemannian double cover (\textit{i.e.}, the locally isometric double cover) of $S^1$. It lifts to a locally isometric double cover $\widetilde{X}$ of $X$, which is close to the fan we want. 

To understand this directly, let $V$ denote the vertices of $X$, and let $V'$ denote a second copy of these vertices. Add geodesic edges
\begin{itemize}
\item From $a$ to $b$ and $a'$ to $b'$ if the corresponding vertices in $X$ are connected by an edge in $X$ that does not intersect 
$\widetilde{L}^+$, and
\item From $a$ to $b'$ and from $a'$ to $b$ if the corresponding vertices in $X$ are connected by an edge in $X$ that intersects $\widetilde{L}^+$ and goes from $a$ in $L^-$ to $b$ in $L^+$.
\end{itemize}
The resulting graph is $\widetilde{X}$.

Unfortunately, the edges are doubled in this cover, so the graph is not embedded. Hence we perturb $\widetilde{X}$ into general position, while the weights of the balancing change continuously as the balancing in $X$ is stable under small perturbations. Since $\dim X<\frac{n-1}{2}$, this ensures we obtain an embedded graph, and positivity is preserved provided the perturbation is small enough.
Let $\widehat{\mathcal{F}}_{2,n}$ denote the balanced fan resulting from coning over the edges of this graph. Finally, $\widetilde{X}\setminus \widetilde{L}^+$ consists of two connected components as we took a double cover, and in particular the map
\[\widetilde{H}^0(\widetilde{X})\ \longrightarrow\ \widetilde{H}^0(\widetilde{X}\setminus \widetilde{L}^+)\]
has a nontrivial cokernel. 

Hence, by Alexander duality, where we interpret $\widetilde{L}^+\cap S^{n-1}$ as a disk by replacing it with a small neighborhood, we obtain that 
\[\widetilde{H}_{n-2}( S^{n-1} \setminus\widetilde{X} )\ \longrightarrow\ \widetilde{H}_{n-2}(S^{n-1}\setminus\widetilde{X},(\widetilde{L}^+ \cap S^{n-1})\setminus\widetilde{X})\]
has a nontrivial cokernel. But then by the long exact sequence in relative homology, we have
\[
\widetilde{H}_{n-3}((\widetilde{L}^+ \cap S^{n-1})\setminus\widetilde{X})\ \longrightarrow\ \widetilde{H}_{n-3}(S^{n-1}\setminus\widetilde{X})
\]
has a nontrivial kernel. The construction is complete, as $S^{n-1}\setminus\widetilde{X}$ is a deformation retract of $\mathbb{R}^n\setminus\widehat{\mathcal{F}}_{2,n}$, and for any $(n-2)$-dimensional  affine space $\ell$ in the interior of $\widetilde{L}^+$ parallel to $\widetilde{\ell},$  $(\widetilde{L}^+ \cap S^{n-1})\setminus\widetilde{X}$ is homeomorphic to $(\mathbb{R}^n\setminus \widehat{\mathcal{F}}_{2,n})\cap \ell.$

\section{Relation to the generalized Hodge conjecture for positive currents}
In this section we discuss that the examples in the preceding section do not in fact satisfy the Hodge-Riemann relations. In particular, following \cite{BH}, they give rise to a family of counter-examples for a (now disproven) Hodge-type conjecture for positive currents introduced by Demailly \cite{DemaillyHodge, DemaillyBook2}, hence generalizing the main result of \cite{BH} to any dimension and codimension greater than $1.$

\begin{conj} The Hodge conjecture for strongly positive currents: 
If $\mathscr{T}$ is a $(p,p)$-dimensional strongly positive closed current on $X$ with cohomology class
\[
\{\mathscr{T}\} \in  \mathbb{R} \otimes_\mathbb{Z} \big(H^{2q}(X,\mathbb{Z})/\textrm{tors}~ \cap H^{q,q}(X) \big),
\]
then $\mathscr{T}$ is a weak limit of the form 
\[
\mathscr{T}=\lim_{i \to \infty} \mathscr{T}_i, \quad \mathscr{T}_i=\sum_j \lambda_{ij} [Z_{ij}],
\]
where $\lambda_{ij}$ are positive real numbers and $Z_{ij}$ are $p$-dimensional subvarieties of $X$.
\end{conj}
\begin{remark}{Demailly in \cite{DemaillyHodge} proved that the above statement implies the Hodge conjecture, and in \cite{DemaillyBook2} observed that if $\mathscr{T}$ in the above statement is assumed to a real current, and $\lambda_{ij}$'s are real numbers, then the statement is equivalent to the Hodge conjecture.}

\end{remark}

\begin{theorem}
For any pair of integers $n\geq 4,$ and $2\leq k \leq n-2,$ 
there is a $n$-dimensional smooth projective toric variety $X$ and a $(k,k)$-dimensional positive closed current $\mathscr{T}$ on $X$ with the following properties:
\begin{enumerate}[(1)]
\item The cohomology class of $\mathscr{T}$ satisfies
\[
\{\mathscr{T}\} \in H^{2(n-k)}(X,\mathbb{Z})\slash\textrm{tors}~ \cap H^{n-k,n-k}(X).
\]
\item The current $\mathscr{T}$ is not a weak limit of currents of the form
\[
\lim_{i \to \infty} \mathscr{T}_i, \quad \mathscr{T}_i=\sum_j \lambda_{ij} [Z_{ij}],
\]
where $\lambda_{ij}$ are nonnegative real numbers and $Z_{ij}$ are $k$-dimensional subvarieties  in $X$.
\end{enumerate}
\end{theorem}
\begin{proof}
The proof goes along the lines of the proof of Theorem 5.1 in \cite{BH}.
\begin{enumerate}
\item[Step 1.] Recall that an extremal current $\TS$ is a current which generates an extremal ray in the cone of closed positive currents, \textit{i.e.}, if $\TS= \TS_1+\TS_2,$ are positive and closed, then for $i=1,2$ there are $\lambda_i>0$, with $\TS_i=\lambda_i \TS_i.$ By application of Milman's converse to the Krein-Milman, the approximability of an extremal current as $
\lim_{i \to \infty} \mathscr{T}_i,$ where $ \mathscr{T}_i=\sum_j \lambda_{ij} [Z_{ij}],$
reduces to the question of approximability of $\TS$ by a sequence of currents $
\lim_{i \to \infty} \mathscr{T}_i,$ where $ \mathscr{T}_i= \lambda_{i} [Z_{i}]$\\
\item[Step 2.] A tropical variety $\mathcal{C}$ is called strongly extremal if \textit{i.e.}, it does not lie in any affine subspace of $\RR^n$, and it can be uniquely weighted up to a multiple to become balanced. It is shown that in  \cite[Theorem 2.12]{BH} that if $\Ccal$ is strongly extremal, then the associated tropical current $\mathscr{T}_{\Ccal}$ is an extremal current. \\

\item[Step 3.] Simple application of El~Mir-Skoda theorem (\cite{DemaillyBook1}) shows that the tropical currents can be extended by zero to any toric variety. Moreover, by \cite[Proposition 4.9]{BH} extremality of tropical currents is preserved after extension {to a \emph{compatible} toric variety, \textit{i.e.}, fibers of $\mathscr{T}_{\Ccal}$ transversally intersect the invariant divisors of the toric variety.} \\ 

\item[Step 4.] By \cite[Theorem 4.7]{BH}, for a $k$-dimensional fan $\FC$, and a compatible smooth projective toric variety $X_{\Sigma}$ the cohomology class of the closed positive current $\{\overline{\mathscr{T}}_{\FC}\} \in H^{k,k}(X_ \Sigma)$ can be presented by $\FC$ itself. By \cite{Gil-Sombra}, existence of such a compatible fan $\Sigma$ is secured. \\
  
\item[Step 5.] Assume that $\{\overline{\mathscr{T}}_{\FC}\}$ is approximable by the currents of the form $\lambda_i  [Z_i],$ {where $\lambda_i$ are nonnegative real numbers and $Z_i$ are irreducible subvarieties. Consider the nef divisors $H_1, H_2, \dots$ on $X$, and the intersection matrix $M= [m_{ij}]$ with}
$$
m_{ij}= \{H_1\}\cup \dots \cup \{H_{p-2}\}\cup  \{\overline{\mathscr{T}}_{\FC}\}. D_i.D_j,
$$
{where each $D_i$ is an invariant divisor of $X_{\Sigma}$ corresponding a ray. Then, $M$ has at most one positive eigenvalue as a result of Hodge index theorem. See \cite[Proposition 5.9]{BH}}.\\ 

\item[Step 6.] By perturbing the divisor class, we can assume that the examples $\widehat{\mathcal{F}}_{k,n}$ of Section 5 are uniquely balanced up to a multiple. As these tropical cycles are stable under small perturbations of their faces, and they preserve their balancing property, we can assume that it also does not lie in any proper affine subspace of $\mathbb{R}^n$. In particular, we obtain that the fans $\widehat{\mathcal{F}}_{k,n}$ can be chosen as strongly extremal. In consequence, $\TS_{\widehat{\mathcal{F}}_{k,n}}$ is an extremal current by Step 2. \\

\item[Step 7.] There exist different natural choices of pullback maps of Chow rings from the fan $\mathcal{F}_X$ over $X$ to the fan {$\widehat{\mathcal{F}}_{2,n}$} over $\widetilde{X}$. For instance, a first choice is to map a degree $1$-class in $\mathcal{F}_X$ to the rays over $V$, and extend linearly to the faces over $V'$. Denote the induced map of Chow rings $\varphi$. 

Symmetrically, we can also choose to map a degree $1$-class in $\mathcal{F}_X$ to the rays over $V'$, and extend linearly to the faces over $V$.
Denote the induced map of Chow rings $\varphi'$. 

Notice that the ample divisors on $\mathcal{F}_X$ pull back to ample divisors on {$\widehat{\mathcal{F}}_{2,n}$} under $\varphi+\varphi'$. We wish to consider the summands separately. Consider therefore an ample divisor in $\mathcal{F}_X$, and pull it back under $\varphi$. Call the resulting divisor $\omega$ in $A^1(\widehat{\mathcal{F}}_{2,n})$. Similarly, let $\omega'$ denote the pullback under $\varphi'$. Let $\text{deg}(a.b)= \mathscr{T}_{\widehat{\mathcal{F}}_{2,n}}.a.b $, and observe that
\[\mathrm{deg}(\omega^2+\omega'^2)=2\mathrm{deg}(\omega^2)=2\mathrm{deg}(\omega'^2)>0,\]
and that $\mathrm{deg}(\omega\omega')=0$ as $\omega$ is linear on $V'$ and $\omega'$ is linear on $V$. 
Hence, the pairing 
\[A^1(\widehat{\mathcal{F}}_{2,n})\times A^1(\widehat{\mathcal{F}}_{2,n})\ \longrightarrow\ \mathbb{R}\]
given by sending two classes $a$, $b$ to
$\mathrm{deg}(ab)$, has two positive eigenvalues. {Accordingly, for}  
$$
\widehat{\mathcal{F}}_{k,n}:= \RR^{k-2} \times \widehat{\mathcal{F}}_{2,n} \subseteq \RR^{n+k-2},
$$
{ the toric variety $(\mathbb{P}^1)^{k-2} \times X_{\Sigma}$ is compatible to $\widehat{\mathcal{F}}_{k,n},$ whenever $X_{\Sigma}$ is compatible to $\widehat{\mathcal{F}}_{2,n}.$ In consequence, the intersection matrix $[(\mathbb{P}^1)^{k-2}.\{\widehat{\mathcal{F}}_{k,n}\}.D_i.D_j]_{ij} $ has two positive eigenvalues. }

\end{enumerate}
To summarize, for each pair of integers $n\geq 4,$ and $2\leq k \leq n-2$, we have constructed a  positive closed extremal current $\TS_{\widehat{\mathcal{F}}_{k,n}}$, which does not satisfy the Hodge index theorem, and the proof is complete. 
\end{proof}

\section{Boundaries of mollified currents}

Andreotti and Grauert in \cite{Andreotti-Grauert} showed that existence of a exhaustion smooth functions $\varphi$ on a complex manifold $X$ implies certain vanishing theorems related to the number of strictly positive eigenvalues for the Leviform of $\varphi$. In particular, the vanishing theorems are implied if $\Omega$ can be approximated by sub-level sets of $\varphi$, which, by Sard's Lemma, can be assumed to have smooth boundaries. Diederich and Fornaess in \cite{Diederich-Fornaess} observed that if one replaces the smooth function $\varphi$ with supremum of a finite number of smooth some weaker vanishing theorem can be deduced as one loses some directions of positivity. 

In the same spirit, interpreting global convexity properties as a vanishing theorem, observing that the global and local notions of convexity are equivalent for domains with smooth boundary, and that the higher convexity properties of open sets passes over the Hausdorff limit, the non-$(n-k)$-convex examples of Section 5, give rise to obstruction for exhaustion of \emph{$k$-pseudoconvex} domain with $k$-pseudoconvex domains which have smooth boundaries. {This is surprising since "sub-level" set of support of any positive smooth form is smooth by Sard's lemma, and there is no restriction on the choice of the mollifiers.} 

We first need the following lemma (See also \cite[Lemma 1.1]{Nisse-Sottile}).

\begin{lemma}\label{useful-lemma}
Suppose that $\{A_i \}$ is a sequence of closed subsets of $\RR^n,$ converging to the closed set $A$ in Hausdorff metric. Let $Z$ be a compact set which is disjoint from $A,$ then there exists a positive integer $N$ such that $A_i \cap Z = \emptyset$ for $i>N.$ 
\end{lemma}
\begin{proof} 
By definition, for every $z\in Z,$ and $\epsilon>0,$ there exists a positive integer $N_z$ such that $d(z, A_i)>\epsilon,$ for $i> N_z.$ Now the statement follows from compactness of $Z$.
\end{proof}

\begin{lemma}\label{k-convergence}
Assume that $\{A_i\}$ is a sequence of subsets of $\RR^n$ with  $A\subseteq A_i$ for all $i,$  converging to $A$ in Hausdorff metric. If the complements $\RR^n \setminus A_i$ are $k$-convex, then $\RR^n \setminus A$ is also $k$-convex.  
\end{lemma}
\begin{proof}
Suppose that $\RR^n \setminus A$ is not $k$-convex, \textit{i.e.}, 
$$
i_L:\widetilde{H}_{k}(L \setminus A)\rightarrow \widetilde{H}_{k}(\RR^n \setminus A), 
$$
is  not injective for an affine $k+1$ dimensional plane $L$. In other words, there exists $0\not=[c]\in \widetilde{H}_k(L\setminus A)$ such that $[c]=\partial [Z],$ for a $(k+1)$ chain $Z$ in $\RR^n \setminus A.$ Since the support of the chain $Z$ is compact, by Lemma~\ref{useful-lemma} we can find a positive integer $N$ that for $i>N,$ $A_i \cap Z = \emptyset.$ Therefore, $[c]= \partial[Z]$ also as an element of $\widetilde{H}_k(\RR^n\setminus A_i).$ Moreover, as $A\subset A_i,$ $[c]\not= 0$ in $\widetilde{H}_k(L\setminus A_i)$ either.  Consequently, $A_i$ cannot be $k$-convex, which is a contradiction.
\end{proof}

Now recall that the standard mollification technique (See for instance \cite[Chapter III, Remark 1.15]{DemaillyBook1} ) of a positive  $(p,p)$ current $\TS$ convoluted with a family of mollifiers ${\mu_t},$ $t\in (0,1),$ yields a family of $(p,p)$ smooth positive forms $\mu_t \star \TS$ such that 
\begin{itemize}
\item $\mu_t \star \TS$ converges to $\TS$ in sense of currents, as $t\to 0$, and 
\item $\text{Supp} (\mu_t \star \TS)$ converges to $\text{Supp} (\TS)$ in Hausdorff sense as $t\to 0.$
\end{itemize}

\begin{theorem}\label{thm-non-moli}
For any pair of integers $n\geq 4,$ and $2\leq k \leq n-2,$ there is a closed positive current $\mathscr{T}\in \mathscr{D}'_{k,k}((\CC^*)^n),$ for which there is no family of mollifiers $\mu_t$ with 
\begin{itemize}
\item[(i)] $\text{Supp} (\mu_t \star \mathscr{T})$ converges to $\text{Supp} (\mathscr{T})$ in 
 sense as $t\to 0$, and
\item [(ii)] $\text{Supp} (\mu_t\star \mathscr{T})$ have smooth boundaries. 
\end{itemize}
\end{theorem}

\begin{proof}
For such $n$ and $k$ consider $\mathscr{T}= \TS_{\widehat{\mathcal{F}}_{k,n}},$ the $(k,k)$-dimensional tropical current associated to the balanced fan $\widehat{\mathcal{F}}_{k,n}$ constructed in Section~\ref{sec-counterexample}. Suppose by contradiction that both statements of the theorem hold. As $\text{Supp} (\mu_t\star \mathscr{T})$ is of the form $\Log^{-1}(A)$ for some closed set $A \subseteq \RR^n$, then by Lemma~\ref{FS-concave} and Corollary~\ref{log-pseudo}, $A$ has no $(n-k)$-caps. Since for each $t$ the boundary of $\text{Supp} (\mu_t\star \TS): =\Log^{-1}(A_t)$ is smooth, by Proposition~\ref{local-global} the complements $\RR^n \setminus A_t$ are $(n-k)$-convex, and by Lemma~\ref{k-convergence} $\RR^n \setminus \widehat{\mathcal{F}}_{k,n}$ must be $(n-k)$-convex, which is a contradiction. 
\end{proof}

\providecommand{\bysame}{\leavevmode\hbox to3em{\hrulefill}\thinspace}
\providecommand{\MR}{\relax\ifhmode\unskip\space\fi MR }
\providecommand{\MRhref}[2]{%
  \href{http://www.ams.org/mathscinet-getitem?mr=#1}{#2}
}
\providecommand{\href}[2]{#2}

\end{document}